\documentclass[a4paper, 10pt]{amsart}
\usepackage[english]{babel}
\usepackage{amsfonts}
\usepackage{amssymb}
\usepackage{marvosym}

\usepackage{graphicx}
\usepackage{scalefnt}
\usepackage{amsmath}
\usepackage{amsthm}

\theoremstyle{plain}
\newtheorem{theo}{Theorem}[section]

\theoremstyle{definition}

\theoremstyle{plain}
\newtheorem{prop}[theo]{Proposition}

\theoremstyle{plain}

\theoremstyle{plain}
\newtheorem{coro}[theo]{Corollary}

\theoremstyle{definition}

\theoremstyle{definition}
\newtheorem{remark}[theo]{Remark}

\begin{document} 

\title{Mal'cev-Neumann Rings and Noncrossed Product Division Algebras}
\author{C\'ecile Coyette}
\address{Universit\'e catholique de Louvain, Chemin du cyclotron, 2, 1348 Louvain-la-Neuve, Belgium}
\email{cecile.coyette@uclouvain.be}
\keywords{Malcev-Neumann series, division algebra, (non)crossed product, local invariants}

\thanks{The author is a Research Fellow of the ``Fonds National pour la Recherche Scientifique" (F.R.S.-FNRS), Belgium}

\begin{abstract}
The first section of this paper yields a sufficient condition for a Mal'cev-Neumann ring of formal series to be a noncrossed product division algebra.
This result is used in Section 2 to give an elementary proof of the existence of noncrossed product division algebras (of degree $8$ or degree $p^2$ for $p$ any odd prime). The arguments are based on those of Hanke in [7], [8] and~[9].
\end{abstract}

\maketitle

A finite-dimensional central simple algebra $A$ over a field $F$ is a crossed product if it contains a maximal commutative subalgebra $K$ that is a Galois field extension of $F$. 
Multiplication of the elements in a $K$-base of $A$ can be defined in terms of a factor set, which yields a nice explicit representation of $A$, see [13, \S14.1]. 
In 1972, S.~A.~Amitsur proved by a generic construction the existence of noncrossed product division algebras [2]. Since then, various examples of noncrossed product division algebras have been constructed [3],[4], including some by E. S. Brussel~[3] over rational function fields or Laurent series fields in one indeterminate over $\mathbb{Q}$. 
More recently, very explicit, computational examples over Laurent series fields were given by T.~Hanke [7], [8] and~[9]. 

The purpose of this article is to provide a new proof of the existence of noncrossed product division algebras of degree $8$ or $p^2$ for $p$ an odd prime. The proof is directly inspired by T. Hanke's explicit constructions. For Hanke's delicate definition of outer automorphisms of division algebras in [7], [8] and [9], we substitute an argument using local invariants of division algebras over global fields. This is made possible by the use of a version of the Mal'cev-Neumann construction. Thus we obtain an elementary but not completely explicit construction of noncrossed product division algebras.

In the first part of this article, we provide a condition for a Mal'cev-Neumann ring to be a division algebra and some sufficient conditions for this Mal'cev-Neumann division algebra to be a noncrossed product, see Proposition \ref{lemme} and Corollary \ref{crit}. The residue field condition used in Proposition \ref{lemme} was first introduced by E. S. Brussel [3].

In the second part, using Corollary \ref{crit}, we obtain two distinct constructions of noncrossed product division algebras of degree $p^2$ (for $p$ an odd prime) and of degree $8$ (respectively Theorem \ref{theo1} and \ref{cas2}). Theorem \ref{theo1} is based on Lemma 7.6 in [7] and [9] and Theorem \ref{cas2} is a generalisation of Theorem 6.2  in [8].


As an illustration of this technique, we give a infinite series of examples based on number field extensions of characteristic $0$ in \S 2.1 and on extensions of rational function fields in one indeterminate in \S 2.2 (these last examples in prime characteristic differ from~[8] where T. Hanke use number fields).

\vspace{3mm}

\section{Malcev-Neumann series}

\vspace{3mm}

Let $A$ be a finite-dimensional central simple $F$-algebra, where $F$ is a field. Assume that the algebra $A$ contains a Galois extension $K/F$ whose Galois group $G := Gal[K/F]$ is abelian and let $C_{A}(K)$ be the centralizer of $K$ in $A$.
Moreover, let $\varepsilon : \Gamma \rightarrow G$ be a surjective homomorphism, where $\Gamma$ is a totally ordered abelian (additive) group.

\vspace{2mm}

Using the Skolem-Noether theorem [13, \S12.6], we know that for all $\sigma \in G$, there is $u_{\sigma} \in A^{\times}$ such that
$x^{\sigma} = u_{\sigma}^{-1}xu_{\sigma}$ for all $x \in K$ (to simplify, we take $u_{\rm Id}=1$).

\vspace{2mm}

From this, we construct the \emph{Malcev-Neumann ring} $C_{A}(K)((\Gamma))$ of formal series in indeterminates $a_{\gamma}$, 
\begin{align*}
&\sum_{\gamma \in \Gamma} a_{\gamma}r_{\gamma} && \text{with $r_{\gamma}\in C_{A}(K)$ for all $\gamma \in \Gamma$,}\\
\intertext{\noindent
such that the set ${\rm supp}\big(\sum_{\gamma \in \Gamma}a_{\gamma}r_{\gamma}\big):=\{\gamma \in \Gamma \, \vert \, r_{\gamma}\neq 0\},$
called the \emph{support} of the series, is well-ordered.
The ring structure is given by componentwise addition and by a multiplication defined as follows: }
ra_{\gamma}&=a_{\gamma}(u_{\varepsilon(\gamma)}^{-1}ru_{\varepsilon(\gamma)}) &&\text{for all  } r \in C_{A}(K), \\
a_{\gamma}a_{\delta}&=a_{\gamma + \delta}u_{\varepsilon(\gamma)\varepsilon(\delta)}^{-1}u_{\varepsilon(\gamma)}u_{\varepsilon(\delta)} &&
\text{for all $\gamma$, $\delta \in \Gamma$.}
\end{align*}

The fact that the support is well-ordered implies that the multiplication is well-defined (see [11, p. 241-244] and [5, \S2.4] for more details). It is easy to prove that the multiplication is associative.
If $L \subseteq C_{A}(K)$ and if $\Lambda$ is a subset of $\Gamma$, the notation $L((\Lambda))$ will refer to the subset of $C_{A}(K)((\Gamma))$ given by the elements 
$\sum_{\gamma \in \Lambda} a_{\gamma} r_{\gamma}$ where $r_{\gamma} \in L$ for all $\gamma \in \Lambda$.

\vspace{2mm}

\begin{remark}\label{rem}
Suppose $\Gamma = \mathbb{Z}$ and let $K/F$ be a cyclic Galois extension with Galois group $G$ generated by $\sigma$. Therefore, for all $\tau \in G$, we have that $\tau=\sigma^{k}$ for some $k\in \mathbb{Z}$. Let $A$ be a finite-dimensional central simple $F$-algebra containing $K$. If we take $u_{\sigma^{z}}=u_{\sigma}^{z}$ and $\varepsilon : \mathbb{Z}\rightarrow G$ defined by $\varepsilon(z)=\sigma^{z}$ for $z \in \mathbb{Z}$ in the above construction, we have $a_{\gamma}a_{\delta}=a_{\gamma+\delta}$ and $ra_{\gamma}= a_{\gamma}r^{\sigma^{\gamma}}$for all $\gamma, \delta \in \mathbb{Z}$.
If we identify $a_{1}$ with $x$, and thus $a_{z}$ with $x^{z}$ for all $z\in \mathbb{Z}$, we obtain a twisted Laurent series ring whose elements are $\sum_{z \in \mathbb{Z}}x^{z}r_{z}$ with $r_{z} \in C_{A}(K)$.
\end{remark}

\vspace{2mm}

To simplify the notation, we set $\mathcal{D}:=C_{A}(K)((\Gamma))$.

\vspace{2mm}

\begin{theo}\label{DivAl}
If $C_{A}(K)$ is a division algebra, then $\mathcal{D}$ is a valued division algebra with value group $\Gamma$ and residue field $C_A(K)$. Its center is $F((\ker \varepsilon))$ and $\deg \mathcal{D} = \deg A$.
\end{theo}

\begin{proof}
Firstly, we prove that the center $Z(\mathcal{D})$ of $\mathcal{D}$ is $F((\ker \varepsilon))$.
Let $S:= \sum_{\gamma\in \Gamma} a_{\gamma}r_{\gamma}$ be an element of $Z(\mathcal{D})$. From the fact that $rS=Sr$ for all $r \in C_{A}(K)$, we conclude that 
$ r u_{\varepsilon(\gamma)}r_{\gamma} = u_{\varepsilon(\gamma)} r_{\gamma} r$, for all 
$r \in C_{A}(K)$ and for all $\gamma \in \Gamma$. Thus, for all $\gamma \in \Gamma$, $u_{\varepsilon(\gamma)} r_{\gamma}$ is an element of $C_{A}(C_{A}(K))=K$.
In a second step, we know that for all $\beta \in \Gamma$, we must have $Sa_{\beta} = a_{\beta}S$. Using the commutativity of $\Gamma$ and~$G$, this condition amounts to $u_{\varepsilon(\beta)}^{-1}u_{\varepsilon(\gamma)}r_{\gamma}u_{\varepsilon(\beta)}=u_{\varepsilon(\gamma)}r_{\gamma}$ for all $\beta, \gamma \in \Gamma$.
Since $u_{\varepsilon(\gamma)}r_{\gamma} \in K$ and since the homomorphism $\varepsilon$ is surjective, $u_{\varepsilon(\gamma)}r_{\gamma}$ is fixed by all the elements of $G$, thus 
$u_{\varepsilon(\gamma)}r_{\gamma} \in F$ for all $\gamma \in \Gamma$. 
In particular, every element $k$ of $K$ commutes with $u_{\varepsilon(\gamma)}r_{\gamma}$. Therefore, we obtain that for all $\gamma \in \Gamma$ and for all $k \in K$, $k^{\varepsilon(\gamma)} r_{\gamma} = kr_{\gamma}$. Fix $\gamma \in \Gamma$.
If $\varepsilon(\gamma) \neq {\rm Id}$, there exists an element $k$ of~$K^{\times}$ such that $k^{\varepsilon(\gamma)}  \neq k$, and this implies that $r_{\gamma} =0$.
We conclude that if $r_{\gamma} \neq 0$ then $\varepsilon(\gamma) = {\rm Id}$.
From these conditions, we find that $Z(\mathcal{D})=F((\ker \varepsilon))$.

\vspace{2mm}

Secondly, in order to prove that $\deg \mathcal{D} = \deg A$, we start with the equality 
$[\mathcal{D} : Z(\mathcal{D})] = (\Gamma : \ker \varepsilon) [C_{A}(K):F].$ 
The homomorphism $\varepsilon$ is surjective, hence 
$(\Gamma : \ker \varepsilon)= \vert G \vert = [K : F]$. From the theory of central 
simple $F$-algebras, we have $[C_{A}(K) : F][K : F]=[A : F]$ (see [13, \S12.7]). Consequently, we obtain that $[\mathcal{D} : Z(\mathcal{D})] = [A : F]$ as desired.

\vspace{2mm}

To prove that $\mathcal{D}$ is a division algebra, we introduce a map $v : \mathcal{D} \rightarrow \Gamma \cup \{ \infty \}$ 
defined by $v(0)= \infty$ and
$$v\big(\sum_{\gamma \in \Gamma}a_{\gamma}r_{\gamma}\big)=\min \bigg\{{\rm supp}\big(\sum_{\gamma \in \Gamma}a_{\gamma}r_{\gamma}\big)\bigg\},$$
This map satisfies the following three properties:
\begin{enumerate}
\item[{$\bullet$}] For $S \in \mathcal{D}$, $v(S)= \infty$ if and only if $S=0$.
\item[{$\bullet$}] If $S_{1}, S_{2} \in \mathcal{D}$, then $v(S_{1}+S_{2})\geq \min \{v(S_{1}), v(S_{2})\}$.
\item[{$\bullet$}] For $S_{1}, S_{2} \in \mathcal{D}$, $v(S_{1}S_{2})=v(S_{1})+v(S_{2})$.
\end{enumerate}

\noindent The proofs of the first two properties are obvious. For the third one, consider two elements $S_{1}:=\sum_{\gamma \in \Gamma} a_{\gamma}r_{\gamma}$ and $S_{2}:=\sum_{\sigma \in \Gamma} a_{\sigma}s_{\sigma}$ in $\mathcal{D}$. We easily see that we have $v(S_{1}S_{2}) \geqslant v(S_{1})+v(S_{2})$. For the reverse inequality, consider the coefficient of $a_{v(S_1)+v(S_{2})}$ in $S_{1}S_{2}$, which is $\alpha:=u_{\varepsilon(v(S_{1}))\varepsilon(v(S_{2}))}^{-1}u_{\varepsilon(v(S_{1}))} r_{v(S_{1})} u_{\varepsilon(v(S_{2}))}s_{v(S_2)}$. Since $r_{v(S_{1})}\neq 0$ and $s_{v(S_2)} \neq 0$, the fact that $C_{A}(K)$ is a division algebra implies that $\alpha \neq 0$, thus $v(S_{1}S_{2}) \leqslant v(S_1)+v(S_2)$. This proves the third property.

\vspace{2mm}

\noindent Therefore, $\mathcal{D}$ has no zero divisors, and since $\mathcal{D}$ is finite-dimensional, each $S \in \mathcal{D}^{\times}$ has an inverse. This concludes the proof: the algebra $\mathcal{D}$ is a division algebra and $v$ a valuation. The residue field of $\mathcal{D}$ for $v$ is $C_A(K)$ and its value group is $\Gamma$.
\end{proof}

\vspace{2mm}

\begin{remark}
In this text, when we use a valuation on a Malcev-Neumann division algebra, it will always be this valuation $v$. The residue field of $\mathcal{D}$ for this valuation is~$C_A(K)$. Therefore, for any subfield $L \subseteq \mathcal{D}$, the residue field~$\overline{L}$ centralizes $K$, and we may consider the field compositum $\overline{L}K \subseteq C_{A}(K)$.
\end{remark}

\vspace{3mm}

The first part of the following lemma on the Mal'cev-Neumann ring is a special case of Corollary 5.16 in [7] and a generalization of Lemma~5.3 in [8] (where $\Gamma=\mathbb{Z}$) and  of Lemma~2 in [3]. Another proof can be found in [10], Theorem 5.15b.

\vspace{2mm}

\begin{prop}\label{lemme}
If $C_{A}(K)$ is a division algebra and if the field $L$ is a maximal subfield of $\mathcal{D}$ then $\overline{L}K$ is a maximal subfield of $C_{A}(K)$. Moreover, if $L/ F((\ker \varepsilon))$ is a Galois extension and if ${\rm char}\,F$ does not divide $[\overline{L} : F]$, $\overline{L}K$ is Galois over $F$.
\end{prop}

\begin{proof}
Since $\overline{L}K$ is a subfield of $C_{A}(K)$, it follows that $\deg C_{A}(K) \leqslant [C_{A}(K) :\overline{L}K].$
Consequently, we have to prove that $[C_{A}(K) : \overline{L}K] \leqslant \deg C_{A}(K)$. 

\vspace{2mm}

\noindent {\it Step 1: $[\mathcal{D} : L]=n \deg C_{A}(K)$ where $n:= [K : F]$.}

\noindent Since $L$ is a maximal subalgebra of $\mathcal{D}$ and using Theorem \ref{DivAl},
$$[\mathcal{D} : L]^{2}=(\deg \mathcal{D})^{2} = [A : F]=(\deg C_{A}(K))^{2}[K : F]^{2}.$$
Thus, $[\mathcal{D} : L] = n \deg C_{A}(K)$ and this concludes the proof. 

\vspace{2mm}

\noindent {\it Step 2: $[C_{A}(K) : \overline{L}] \leqslant {n \over l} \deg C_{A}(K)$ where $l :=\big(v(\mathcal{D}^{\times}\big) : v(L^{\times}))$.}

\noindent This readily follows from Step 1 and from the fundamental inequality of valuation theory saying that 
$$\big( v(\mathcal{D}^{\times}) : v(L^{\times})\big) [C_{A}(K) : \overline{L}] \leqslant [\mathcal{D} : L].$$

\vspace{2mm}

\noindent {\it Step 3: For the same $l$ as above, $[K : \overline{L} \cap K] \geqslant {n \over l}$.}

\noindent Indeed, the surjective homomorphism $\varepsilon$ induces an isomorphism
$$\varepsilon : \Gamma/ \ker \varepsilon \rightarrow Gal[Z(\overline{\mathcal{D}})/\overline{Z(\mathcal{D})}] = Gal[K/F].$$
Let $x$, $y$ be two elements of $L$ with $v(y)=0$ and $\bar{y}\in \overline{L} \cap K$. If $v(x)= \gamma$, then $x=a_{\gamma}r(1+S)$ with $v(S)>0$ and $r \in C_{A}(K)$. Therefore, 
$$\overline{y}=\overline{x^{-1}yx} = \overline{r}^{-1}\, \overline{a_{\gamma}^{-1}ya_{\gamma}}\, \overline{r}.$$
Using the multiplication rules, we obtain $\bar{y}^{\varepsilon (\gamma)} = \bar{y}$.
This implies that $\varepsilon(v(L^{\times}))$~is contained in $Gal[K/\overline{L} \cap K]$. Since $(v(\mathcal{D}^{\times}) : v(Z(\mathcal{D})^{\times}))=(\Gamma : \ker \varepsilon)=[K : F] =n$, we know that $(v(L^{\times}) : \ker \varepsilon) = {n \over l}$ divides $[K : \overline{L} \cap K]$. Consequently, 
we conclude immediately that ${n \over l} \leqslant [K : \overline{L} \cap K]$.

\vspace{2mm}

\noindent {\it Step 4: $[C_{A}(K) : \overline{L}K] \leqslant \deg C_{A}(K)$.}

\noindent From Step 3, it follows that $[\overline{L}K : \overline{L}] \geqslant {n \over l}$. Hence, since $$[C_{A}(K) : \overline{L}K] = {[C_{A}(K) : \overline{L}] \over [\overline{L}K : \overline{L}]},$$
and using Step 2, we obtain $[C_{A}(K) : \overline{L}K] \leqslant \deg C_{A}(K)$.

\vspace{3mm}

Suppose now that $L/F((\ker \varepsilon))$ is a Galois extension. It follows that $\overline{L}/F$ is a normal extension (see [6, p. 136]). Moreover, $\overline{L}/F$ is separable because ${\rm char}\,F$ does not divide $[\overline{L} : F]$, hence $\overline{L}/F$ is a Galois extension.
 Therefore, since $K/F$ and $\overline{L}/F$ are Galois extensions, $\overline{L}K/F$ is a Galois extension (of degree $\deg A$).
 \end{proof}

\vspace{3mm}

\begin{remark}
For $C_A(K)$ a division algebra, if a field $M$ is a maximal subfield of~$C_{A}(K)$, dimension count shows that $M$ is also a maximal subfield of $A$.  
\end{remark}

\vspace{3mm}

The following result (similar to Theorem 5.20 in [7]) provides a condition for $\mathcal{D}$ to be a noncrossed product.

\vspace{3mm}

\begin{coro}\label{crit}
Let $A$ be a central division $F$-algebra containing a Galois extension~$K/F$ with a abelian Galois group $G:=Gal[K/F]$ and such that ${\rm char}\,F$ does not divide $\deg A$. Suppose $\varepsilon : \Gamma \rightarrow G$ is a surjective homomorphism, where $\Gamma$ is a totally ordered abelian (additive) group.
If $A$ contains no maximal subfield $M$ containing $K$ such that $M/F$ is a Galois extension, then $\mathcal{D}$ is a noncrossed product division algebra.
\end{coro}

\begin{proof}
We have the hypotheses needed to construct the Mal'cev-Neumann ring $\mathcal{D}$. 
Since $A$ is a division algebra, the centralizer $C_{A}(K)$ is clearly a division algebra. Consequently, Theorem \ref{DivAl} says that $\mathcal{D}$ is a division algebra.
If $\mathcal{D}$ is a crossed product, there exists a maximal subfield $L$ of $\mathcal{D}$ such that $L/F((\ker \varepsilon))$ is a Galois extension. The fact that ${\rm char}\,F$ does not divide $\deg A$ implies that ${\rm char}\,F$ does not divide $[\overline{L}:F]$.
By Proposition \ref{lemme}, $M:= \overline{L}K$ is a maximal subfield of $C_{A}(K)$, and thus a maximal subfield of $A$, such that $\overline{L}K$ Galois over~$F$. The  conclusion is clear: by hypothesis, $A$ contains no such subfield. 
\end{proof}
 
\vspace{3mm}

\theoremstyle{plain}
\newtheorem{theo2}{Theorem}[section]

\theoremstyle{definition}
\newtheorem{defi2}[theo2]{Definition}

\theoremstyle{plain}
\newtheorem{prop2}[theo2]{Proposition}

\theoremstyle{plain}
\newtheorem{lemma2}[theo2]{Lemma}

\theoremstyle{plain}
\newtheorem{coro2}[theo2]{Corollary}

\theoremstyle{definition}
\newtheorem{example2}[theo2]{Example}

\theoremstyle{definition}
\newtheorem{remark2}[theo2]{Remark}

\section{Application to noncrossed product division algebras}

\vspace{3mm}

In the following two subsections, we consider two particular cases of central division $F$-algebras $A$ containing a Galois extension $K/F$. We suppose to have the general hypotheses for Corollary \ref{crit} and, in both cases, we discuss now which conditions we must add for $A$ to have no maximal subfield $M$ containing $K$ such that $M/F$ is Galois.

\vspace{3mm}

\subsection{Noncrossed product division algebra of degree $p^{2}$, for $p$ an odd prime}
\mbox{}

\vspace{3mm}

\noindent The following theorem gives some general conditions under which Corollary \ref{crit} applies (for $K/F$ a global fields extension, it is a special case of Lemma 7.6 in [7]). 

\vspace{2mm}

\begin{theo2}\label{theo1}
Let $A$ be a central division $F$-algebra of degree $p^{2}$ containing a Galois extension $K/F$ of degree $p$, where $p$ is an odd prime.
Suppose we have two discrete valuations $v_1$, $v_2$ of $F$ such that: 
\begin{enumerate}
\item $v_1$, $v_2$ extend to valuations of $A$,
\item $K/F$ is totally ramified for $v_1$, ${\rm char}\,\overline{F}_{v_1}\neq p$ and  $\mu_{p^{2}} \not\subset \overline{F}_{v_{1}}$ (i.e. $\overline{F}_{v_{1}}$ contains no primitive $(p^{2})^{th}$-root of unity).
\item $K/F$ is inertial for $v_2$, $\overline{F}_{v_2}$ is finite and $\vert\overline{F}_{v_{2}}\vert \not\equiv0,1\mod p$.
\end{enumerate}
Then, the algebra $A$ has no maximal subfield $M$ containing $K$ such that $M/F$ is a Galois extension. Consequently, $\mathcal{D}$ is a noncrossed product division algebra.
\end{theo2}

\begin{proof}
Suppose we have a maximal subfield $M$ of $A$ containing the field $K$ and such that~$M/F$ is a Galois extension. Since $[M:F]=p^2$, we have two possibilities. 
By hypothesis (1), the valuations $v_1$ and $v_2$ extend uniquely to the field $M$ (see [15]). 

\vspace{2mm}

First case: $M/F$ is a cyclic extension.

Since $M/F$ is cyclic, $K$ is the only subfield of $L$ of dimension $p$. The valuation $v_1$ is totally ramified in $K$, hence the inertia field of $v_1$ for $K/F$ is $F$. Therefore, since the subfields of $M$ are linearly ordered, the inertia field of  $v_1$ for $M/F$ is also the field $F$. Consequently, $v_1$ is totally ramified in $M$.

For the valuation $v_1$, the field extension $M/F$ is tamely ramified: $\overline{M}_{v_1} = \overline{F}_{v_1}$ and $([M:F], {\rm char}\,\overline{F}_{v_1})=1$. Since $M/F$ is a cyclic tamely and totally ramified extension for a discrete valuation, we conclude that $\overline{F}_{v_1}$ contains a primitive $(p^{2})^{\rm th}$-root of unity, using Proposition $(\ast)$ in [9, p. 201] (or Theorem 16.2.6 in [6]). This is a contradiction with the hypothesis (2). Thus, $K$ has no cyclic extension of degree~$p^{2}$ in $A$.

\vspace{2mm}

Second case: $M/F$ is not a cyclic extension, hence $M/F$ is an elementary abelian extension.

In this case, $M$ has a subfield $K'$ with $K \not=K'$ such that $[K' : F]=[K:F]=p$. Since $\overline{F}_{v_2}$ is finite, the valuation $v_2$ is inertial in $K$ implies that it is totally ramified in $K'$. Furthermore, $\overline{K'}_{v_2}=\overline{F}_{v_2}$ and, using $\vert\overline{F}_{v_{2}}\vert \not\equiv0\mod p$, we obtain that $([K' : F], {\rm char}\,\overline{F}_{v_2})=1$. Hence $K'/F$ is tamely ramified. 
The extension $K'/F$ is cyclic, hence $\overline{F}_{v_2}$ contains a primitive $p^{\rm th}$-root of unity (using the same result as in the cyclic case, in [9]), thus $\vert\overline{F}_{v_{2}}\vert \equiv1\mod p$. This contradicts the hypothesis~(3).  

\vspace{2mm}

In conclusion, such a field $M$ does not exist, hence $\mathcal{D}$ is a noncrossed product division algebra.
\end{proof}

\vspace{3mm}

To construct a noncrossed product division algebra using Theorem \ref{theo1}, it remains to have an algebra and a field extension satisfying all the hypotheses we need to use the theorem.
To do this, we assume that $F$ is a global field.
Corollary \ref{invloc} shows that for any Galois extension $K/F$ of odd prime degree $p$ satisfying properties (2) and~(3), there always exists a central division $F$-algebra $A$ of degree~$p^2$ containing~$K$ and satisfying (1). The proof uses local invariants. For more details about this subject, see [13,~\S17.10,~\S18.4,~\S18.5] and [14, \S8.32]. 

\noindent We will write $V(F)$ to designate the set of all primes of $F$ (see [14, p. 63-64]) and ${\rm Inv}_{v}(A)$ for the local invariant of $A$ for the valuation $v$.

\vspace{2mm}

\begin{prop2}\label{contain}
Let $A$ be a central division algebra over a global field $F$ of degree~$n$. An extension $K$ of $F$ of degree $k$ dividing $n$ is contained in $A$ if and only if the Least Common Multiple (LCM) of orders of ${\rm Inv}_{w}(A \otimes_{F} K)$ for $w \in V(K)$ is $n/k$.
\end{prop2}

\begin{proof}
We know that there exists an $F$-algebra embedding of $K$ in $A$ if and only if ${\rm Ind}(A\otimes_{F} K)[K : F]={\rm Ind}\,A$ (for the proof, see Theorem 24 in [1, p. 61]).
Then, the hypotheses implies that $K$ is contained in $A$ if and only if ${\rm Ind}(A \otimes_{F} K)=n/k$. Since ${\rm Ind}(A\otimes_{F} K)$ is the LCM of the local indices (see [14, p. 279]), and since the local indices are equal to the orders of the local invariants ${\rm Inv}_{w}(A\otimes_{F}K)$ (see [13,~\S17.10]), we conclude the proof. 
\end{proof}

\vspace{2mm}

\begin{coro2}\label{invloc}
Let $n$ be a positive integer and let $K/F$ be a global field extension of degree $k$ such that $n=km$ for some positive integer $m$. Suppose $v_1$, $v_2 \in V(F)$ extend uniquely to $K$.
Let $S$ be a finite subset of $V(F)$ containing $v_1$ and $v_2$. Consider a sequence $(t_v)_{v \in V(F)}$ of elements of $\mathbb{Q}/\mathbb{Z}$ such~that:
\begin{enumerate}
\item [{(a)}] $\sum_{v\in V(F)}t_{v}=0+ \mathbb{Z}$,
\item [{(b)}] the order of $t_v$ is $n$ for $v \in \{v_1, v_2\}$,
\item [{(c)}] the order of all nonzero $t_v$ divides $m$,
\item [{(d)}] $t_v=0$ for all $v \in V(F)\backslash S$.
\end{enumerate}
Then there exists a central division $F$-algebra $A$ with local invariants $(t_{v})_{v \in V(F)}$. Moreover $A$ has degree $n$, contains $K$ and the valuations $v_1$ and $v_2$ extend to $A$.
\end{coro2}

\begin{proof}
Since we have $\sum_{v \in V(F)}t_{v}=0+ \mathbb{Z}$, the Hasse-Brauer-Noether Theorem (see [14, p. 277]) shows that there exists a unique central division $F$-algebra $A$ such that ${\rm Inv}_{v}(A)=t_{v}$ for all $v \in V(F)$. 
The LCM of the orders of invariants is~$n$, this implies that this division algebra $A$ has degree $n$. 

For a field $F$, we write~$\hat{F}_{v}$ to designate the completion of $F$ for the valuation $v$.
For $i=1$ and $i=2$, the order of ${\rm Inv}_{v_i}(A)$ is~$n$, hence, $A\otimes \hat{F}_{v_i}$ is a division algebra. Consequently, $v_i$ extend to $A\otimes \hat{F}_{v_i}$ for $i\in \{ 1, 2\}$. By restriction, the valuations $v_{1}, v_{2}$ extend to $A$.

We must now prove that the field~$K$ is contained in $A$. The valuations $v_{1}$ and~$v_{2}$ extend uniquely to $K$, hence, $[\hat{K}_{v_{i}} : \hat{F}_{v_{i}}]= [K : F]$ respectively for $i=1$ and $i=2$.
For all $v \in V(F)$, we know that 
\begin{equation}\label{equation}
{\rm Inv}_{w}(A\otimes_{F}K)=[\hat{K}_{w}:\hat{F}_{v}]\,{\rm Inv}_{v}(A)
\end{equation} 
for any extension $w$ of the valuation $v$ to~$K$ (see [14, p. 354]). Then, the order of ${\rm Inv}_{v_{i}}(A \otimes_{F} K)$ is $m$, for $i \in \{1, 2\}$. Moreover, for all valuations $v \in V(F)$, since the order of ${\rm Inv}_{v}(A)$ divides $m$, the order of ${\rm Inv}_{w}(A\otimes_{F}K)$ divides $m$, using~(\ref{equation}). Therefore, it follows that $A$ contains~$K$ by Proposition \ref{contain}. 
\end{proof}

\vspace{2mm}

In conclusion, for a given global field Galois extension~$K/F$ of prime degree $p$ such that the valuations $v_1$ and $v_2$ of $F$ extend uniquely to $K$, we have proved the existence of an algebra satisfying the properties we need.


\vspace{3mm}

To have an example of noncrossed product algebra with this construction, we must consider a global field Galois extension $K/F$ of prime degree, satisfying conditions (2) and (3) in Theorem \ref{theo1}. An explicit example of degree $3$ is given by Hanke in~[9]. Hanke's example is the Galois extension $\mathbb{Q}(\zeta+\zeta^{-1})/\mathbb{Q}$ where~$\zeta$ is a primitive $7^{\rm th}$-root of unity. In the following construction, we will provide an example of degree $p$ for any odd prime $p$.

\vspace{2mm}

\begin{example2} \label{ex2.4}
Let $p$ be an arbitrary odd prime. By Dirichlet's theorem on primes in arithmetic progression, there exists a prime $q\equiv1+p \mod p^2$. Let $\zeta_q \in \mathbb{C}$ be a primitive $q$-th root of unity. Since $\mathbb{Q}(\zeta_q)$ is a cyclic extension of $\mathbb{Q}$ of degree $q-1\equiv 0 \mod p$, it contains a unique cyclic extension $K$ of $\mathbb{Q}$ of degree $p$. Choose $m \in \mathbb{Z}$ whose residue modulo $q$ generates the multiplicative group $(\mathbb{Z}/q\mathbb{Z})^{\times}$. By Dirichlet's theorem, there exists a prime $r \equiv 2q + m (1-q) \mod pq$.



\begin{prop2}\label{propexo}
The field $K$ of Example \ref{ex2.4} satisfies conditions (2) and (3) of Theorem \ref{theo1} for $F=\mathbb{Q}$, $v_1$ the $q$-adic valuation $v_q$ and $v_2$ the $r$-adic valuation $v_r$.
\end{prop2}

\begin{proof}
Firstly, we know that the extension $\mathbb{Q}(\zeta_q)/\mathbb{Q}$ is totally ramified for the $q$-adic valuation $v_q$. Therefore, the extension $K/\mathbb{Q}$ is totally ramified for $v_q$. The residue field $\mathbb{F}_q$ has characteristic $q\neq p$. Since $q\equiv 1+p \mod p^2$, we have that $q\not\equiv 1 \mod p^2$, thus $\mathbb{F}_q$ contains no primitive $(p^2)^{\rm th}$-root of unity.

Secondly, we consider $K/F$ with the valuation $v_r$. The residue field of $F$ for $v_r$ is~$\mathbb{F}_r$. It is a finite field, and $r \not \equiv 0,1 \mod p$ since $r \equiv 2 \mod p$. It remains to prove that $K/F$ is inertial for $v_r$. Since $m$ is a primitive $(q-1)$-root of unity in $\mathbb{F}_q$ and since $r \equiv m \mod q$, it follows that $r^k \not\equiv 1 \mod q$ for all $k < q-1$. Therefore, no extension of $\mathbb{F}_{r}$ of degree $k<q-1$ contains a primitive $q$-th root of unity. It follows that $v_r$ has a unique extension to $\mathbb{Q}(\zeta_q)$ and the residue field is an extension of $\mathbb{F}_r$ of degree $q-1$. In conclusion, $\mathbb{Q}(\zeta_q)/\mathbb{Q}$ is an inertial extension for the valuation~$v_r$.
\end{proof}

\begin{coro2}
For any odd prime $p$, there exists a noncrossed product division algebra of degree $p^2$ over the Laurent series field in one indeterminate over $\mathbb{Q}$.
\end{coro2}

\begin{proof}
From the prime $p$, we construct the field extension $K/\mathbb{Q}$ in Example \ref{ex2.4}. 
Let $v_1$ be the $q$-adic valuation $v_q$ and $v_2$ be the $r$-adic valuation $v_r$. Consider local invariants defined by 
$$t_{v_{1}}:={p^{2}-1 \over p^{2}} + \mathbb{Z}, \qquad t_{v_{2}}:={1 \over p^{2}}+\mathbb{Z}$$
and $t_{v}:=0+\mathbb{Z}$ for all $v \in V(\mathbb{Q})\backslash \{v_1, v_2\}$. This sequence satisfies Corollary \ref{invloc}, hence there exists a division $\mathbb{Q}$-algebra $A$ of degree $p$ containing $K$ and such that $v_q$ and $v_r$ extend to $A$.
Using Proposition \ref{propexo}, we know that the hypotheses (2) and (3) in Theorem \ref{theo1} are satisfied.
Therefore, for $\Gamma = \mathbb{Z}$ and $\varepsilon : \mathbb{Z} \rightarrow Gal[K/\mathbb{Q}]$ defined by $\varepsilon(z)=\sigma^{z}$ (where $\sigma$ generates $Gal[K/\mathbb{Q}]$), Theorem~\ref{theo1} implies that the twisted Laurent series ring $C_{A}(K)((\mathbb{Z}))$ is a noncrossed product division algebra (see Remark \ref{rem}). 
\end{proof}

\end{example2}

\vspace{3mm}

\subsection{Noncrossed product division algebra of degree $8$}
\mbox{}

\vspace{3mm}

Consider $F$ a field with ${\rm char}\,F\neq 2$, and a field $K:=F(\sqrt{a}, \sqrt{b})$ where $a$,~$b$ are elements in $F^{\times}\backslash F^{\times2}$, independent modulo $F^{\times 2}$. Let $A$ be a central division algebra over its center $F$ such that $\deg A = 8$ and such that $A$ contains $K$.

Let $\sigma$ and $\tau$ be the two automorphisms generating the Galois group of $K/F$, respectively defined by: 
\begin{align*}
\sigma(\sqrt{a})&=-\sqrt{a}, & \sigma(\sqrt{b})&=\sqrt{b},\\
\tau(\sqrt{a})&=\sqrt{a}, &  \tau(\sqrt{b})&=-\sqrt{b}.
\end{align*}
\noindent Define the surjective homomorphism $\varepsilon : \mathbb{Z}\times \mathbb{Z} \rightarrow Gal[K/F]$ by $\varepsilon(r,s)=\sigma^{r}\tau^{s}$. Consider $\mathcal{D}=C_{A}(K)((\Gamma))$ where $\Gamma = \mathbb{Z}\times\mathbb{Z}$.

\vspace{3mm}

Using Theorem \ref{DivAl}, since $A$ is a division algebra, $\mathcal{D}$ is a division algebra too.
By Proposition \ref{crit}, if the central division $A$ contains no maximal subfield~$M$ containing $K$ such that $M/F$ is a Galois extension, then $\mathcal{D}$ is a noncrossed product division algebra.

\vspace{3mm}

We want to find conditions on $K$ and $A$ under which such a Galois extension~$M/F$ of degree $8$ does not exist. 

\vspace{2mm}

The following theorem providing some conditions to obtain a noncrossed product division algebra is a generalization of Lemma 8.2 in [7] and of Theorem 6.2 in [8] (for a general global field $F$ instead of a real number field). An example where these conditions are fulfilled is given in Example \ref{exi} below.

\vspace{2mm}

\begin{theo2}\label{cas2}
Let $F$ be a global field with ${\rm char}\,F \neq 2$. Suppose $A$ is a central division $F$-algebra of degree $8$ containing a biquadratic extension $K:=F(\sqrt{a}, \sqrt{b})$ where $a, b$ are such that the quadratic form $\langle a, b, ab\rangle$ is not isometric to $\langle 1, 1, 1\rangle$. Moreover, assume $v_{1}, v_{2}$ are two nondyadic valuations of $F$ such that:
\begin{enumerate}
\item $v_1, v_2$ extend to valuations of $A$,
\item the inertia fields of $v_1$ and $v_2$ for $K/F$ are two distinct quadratic extensions of $F$,
\item $\vert \overline{F}_{v_1} \vert \not\equiv 1 \mod 4$,
\item $\vert \overline{F}_{v_2} \vert \not\equiv 1 \mod 4$,
\end{enumerate}
then there is no Galois extension $M/F$ of degree $8$ such that $K \subseteq M \subseteq A$. Consequently, the division algebra $\mathcal{D}$ is not a crossed product.
\end{theo2}

\begin{proof}
Assume we have a Galois extension $M/F$ of degree $8$ with $K \subseteq M \subseteq A$. To simplify the notation, we write $G:=Gal[M/F]$.
It follows that $\vert G \vert =8$. We have three possibilities: $G$ is the dihedral group, the quaternion group or an abelian group. 

\vspace{1mm}

The proof that such extension  does not exist when $G$ is the dihedral group or an abelian group is similar to the proof of Theorem 6.2 in [8]. 
The case  where~$G$ is the quaternion group is not possible by Witt's criterion for the embedding of biquadratic extensions into quaternionic extensions: the fact that $\langle a, b, ab \rangle$ and $\langle 1,1, 1\rangle$ are not isometric implies that $Gal[M/F]$ is not a quaternion group (see~[16]).

\vspace{1mm}

\noindent In conclusion, there is no Galois extension $M/F$ of degree $8$ such that $K \subseteq M \subseteq~A$.
\end{proof}

\vspace{3mm}

An explicit example of a noncrossed product division algebra based on the extension $\mathbb{Q}(\sqrt{3}, \sqrt{-7})/\mathbb{Q}$ using iterated twisted Laurent series is given in [7], [8] and can be easily adapted with Mal'cev-Neumann series. In the following example, we consider a biquadratic field extension on a rational function field in one indeterminate. 

\vspace{2mm}

\begin{example2}\label{exi}
Let $p$ be a prime such that $p\equiv 3 \mod 8$. Consider the field extension $K/F$ where $F:=\mathbb{F}_{p}(t)$ (the rational function field in one indeterminate over $\mathbb{F}_{p}$) and $K:=\mathbb{F}_{p}(t)\big(\sqrt{t}, \sqrt{(t+1)(t+2)}\big)$, with the valuations $v_{1}:=v_{t}$ and $v_{2}:=v_{t+1}$ respectively the $t$-adic valuation and the $(t+1)$-adic valuation. 

The valuation $v_{1}$ is totally ramified in $\mathbb{F}_{p}(t)(\sqrt{t})$ and the valuation $v_{2}$ is totally ramified in $\mathbb{F}_{p}(t)\big(\sqrt{(t+1)(t+2)}\big)$. Moreover, we see easily that $\overline{t}$ is not a square in the residue field of $\mathbb{F}_{p}(t)$ for the valuation $v_{t+1}$ and that $\overline{(t+1)(t+2)}$ is not a square in the residue field of $\mathbb{F}_{p}(t)$ for the valuation $v_{t}$ (since $p\equiv 3 \mod 8$). Therefore, the inertia field of $v_1$ for $K/F$ is not equal to the inertia field of $v_2$ for~$K/F$. Consequently, we have the hypothesis (2) in Theorem \ref{cas2}. 

In both cases, for $i \in \{1, 2\}$, $\vert \overline{F}_{v_{i}}\vert = p \equiv 3 \mod 4$, thus the hypotheses (3) and~(4) also hold. 

Furthermore, we want to prove that $\langle t, (t+1)(t+2), t(t+1)(t+2) \rangle \not\simeq \langle 1, 1, 1 \rangle$. To do this, we decompose the quadratic forms as follows:
$$\langle t, (t+1)(t+2), t(t+1)(t+2)\rangle \simeq \langle t \rangle \bot \langle (t+1)(t+2)\rangle \langle 1, t\rangle.$$ 
Consider the valuation $v_{t+2}$ which is ramified in $\mathbb{F}_{p}(t)(\sqrt{(t+1)(t+2)})$. 
Clearly, for this valuation, $\langle \overline{t} \rangle$ is anisotropic. Moreover, $\langle \overline{1}, \overline{t}\rangle$ is anisotropic too since $-\overline{t} \equiv 2$ is not a square in the residue field $\mathbb{F}_{p}$. 
Thus, $\langle t, (t+1)(t+2), t(t+1)(t+2)\rangle$ is anisotropic in $\mathbb{F}_{p}(t)$ (see [12, p. 148]). Since $\langle 1, 1, 1\rangle$ is isotropic in $\mathbb{F}_{p}(t)$ (see [12, p.36]), we have $$\langle t, (t+1)(t+2), t(t+1)(t+2)\rangle \not\simeq\langle 1, 1, 1\rangle.$$ 

\vspace{2mm}

As in \S2.1, we want to prove the existence of an algebra $A$ as in Theorem \ref{cas2} using Corollary \ref{invloc} (with $k=4$ and $n=8$). 
For example, we may choose $$t_{1}:= {3\over 8}+\mathbb{Z} \qquad t_{2}:={5 \over 8}+\mathbb{Z}.$$
Corollary \ref{invloc} implies that there exists a central division $F$-algebra having $t_1$ as local invariant in $v_1$ and $t_2$ as local invariant in $v_2$ and $0$ for all the other local invariants.
Moreover, this central division $F$-algebra $A$ has degree $8$, contains $K$ and $v_1$ and $v_2$ extend to~$A$.

In conclusion, Theorem \ref{cas2} says that $C_{A}(K)((\mathbb{Z} \times \mathbb{Z}))$ is a noncrossed product division algebra. 
\end{example2}

\vspace{5mm}

\noindent {\bf Acknowledgment}

\vspace{2mm}

This paper is based on my masters thesis. I'd like to thank my thesis advisor Jean-Pierre Tignol for his advice and suggestions and the referee for the careful reading and many helpful suggestions to improve the paper. My thanks also go to my family and my friends for their encouragement and support. 

\vspace{10mm}

\noindent {\bf BIBLIOGRAPHY}

\vspace{3mm}

\begin{enumerate}

\item[{[1]}] ALBERT A. A., {\it Structure of Algebras}, (coll: {\it Colloq. Pub.}, 24), Amer. Math. Soc., Providence RI, 1961

\item[{[2]}] AMITSUR S. A., ``On central division algebras", {\it Israel J. Math.}, 12, 1972, p. 408-422

\item[{[3]}] BRUSSEL E. S., ``Noncrossed Products and Nonabelian Crossed Products over $\mathbb{Q}(t)$ and $\mathbb{Q}((t))$", {\it Amer. Jour. Math.}, vol. 117, 1995, p. 377-394

\item[{[4]}] BRUSSEL E. S., ``Noncrossed Products over $k_{\mathfrak{p}}(t)$", {\it Trans. Amer. Math. Soc.}, vol. 353, No 5, 2000, p. 2115-2129

\item[{[5]}] COHN P. M., {\it Skew Fields. Theory of general division rings} (coll: {\it Encyclopedia of mathematics and its applications}, 57), Cambridge University Press, New-York, 1995


\item[{[6]}] EFRAT I., {\it Valuations, Orderings, and Milnor $K$-Theory}, (coll: {\it Mathematical Surveys and Monographs}, vol. 124), American Mathematical Society, Providence, RI, 2006

\item[{[7]}] HANKE T., ``A direct Approach to Noncrossed Product Division Algebras" thesis dissertation, Postdam, 2001

\item[{[8]}] HANKE T., ``An explicit example of a noncrossed product division algebra'', {\it Math. Nachr.}, vol. 251, 2004, p. 51-68

\item[{[9]}] HANKE T., ``A Twisted Laurent Series Ring that is a Noncrossed Product'', {\it Israel Journal of Mathematics}, vol. 150, 2005, p. 199-203


\item[{[10]}] JACOB B., WADSWORTH A. R.,``Division algebras over Henselian Fields", {\it Journal of Algebra}, vol. 128, 1990, p. 126-179 

\item[{[11]}] LAM T. Y., {\it A First Course in Noncommutative Rings} (coll: {\it Graduate Texts in Mathematics}, vol.131), Springer-Verlag, New-York, 1991 

\item[{[12]}] LAM T. Y., {\it Introduction to Quadratic Forms Over Fields} (coll: {\it Graduate Studies in mathematics}, vol. 67), American Mathematical Society, Providence, RI, 2005

\item[{[13]}] PIERCE R. S., {\it Associative Algebras} (coll: {\it Graduate Texts in Mathematics}, vol.~88), Springer-Verlag, New York, 1982

\item[{[14]}] REINER I., {\it Maximal Orders}, Academic Press, London, 1975


\item[{[15]}] WADSWORTH A. R., ``Extending valuations to finite dimensional division algebras'',  {\it Proc. Amer. Math. Soc.}, vol. 98, 1986, p. 20-22

\item[{[16]}] WITT E., ``Konstruktion von galoisschen K\"{o}rpern der Charakteristik $p$ zu vorgegebener Gruppe der Ordnung $p^{f}$", {\it J. reine angew. Math.}, vol. 174, 1936, p. 237-245

\end{enumerate}

\end{document}